\documentclass[a4paper,11pt,reqno]{amsart}

\usepackage{setspace}
\usepackage{geometry} 
\usepackage{amsmath, amsthm, amssymb}
\usepackage{mathtools}      
\usepackage{mathabx}        
\usepackage[bb=fourier,cal=euler,scr=rsfs]{mathalfa}	
\usepackage[shortlabels]{enumitem}       

\usepackage[utf8]{inputenc}

\usepackage{tikz-cd}

\usepackage{comment}

\usepackage{multicol}

\usepackage{transparent}

\usepackage{graphicx}

\newtheorem*{thm*}{Theorem}

\newtheorem{thm}{Theorem}[section]

\newtheorem{lemma}[thm]{Lemma}
\newtheorem{prop}[thm]{Proposition}

\newtheorem{teoA}{Theorem}

\newtheorem{teoAprima}{Theorem}

\theoremstyle{definition}
\newtheorem{definition}[thm]{Definition}

\newtheorem*{quest*}{Question}

\theoremstyle{remark}

\newcommand{\RR}{\mathbb{R}}
\newcommand{\TT}{\mathbb{T}}
\newcommand{\W}{\mathcal{W}}

\DeclareMathOperator{\length}{length}

\usepackage{accents}

\usepackage{xcolor}
\usepackage{hyperref}
\hypersetup{
    colorlinks,
    linkcolor={red!50!black},
    citecolor={blue!50!black},
    urlcolor={blue!80!black}
}

\usepackage{quoting,xparse}

\numberwithin{equation}{section}

\title{Quasi-isometric center action in dimension 3}

\author{Marcielis Espitia} 
\address{Universidad Federal Fluminense, Brasil}
\email{maresno@gmail.com}

\author{Santiago Martinchich} 
\address{IESTA, FCEA, Universidad de la Rep\'ublica, Uruguay}
\email{santiago.martinchich@fcea.edu.uy}

\author{Rafael Potrie} 
\address{Centro de Matem\'atica, Universidad de la Rep\'ublica, Uruguay \& IRL-IFUMI-CNRS}
\email{rpotrie@cmat.edu.uy}
\urladdr{http://www.cmat.edu.uy/~rpotrie/}

\thanks{ M.E. was partially supported by FAPERJ. S.M. was partially supported by Fondo Vaz Ferreira (MEC), PEDECIBA and CSIC. R. P. was partially supported by CSIC I+D project `Estructuras Topol\'ogicas de sistemas parcialmente hiperb\'olicos y aplicaciones'.  }

\begin{document}

\begin{abstract}
We study transitive partially hyperbolic diffeomorphisms in dimension 3 preserving a center foliation on which they act quasi-isometrically. We show that the diffeomorphism is up to finite lift and iterate, either a skew-product or a discretised Anosov flow.

%
\end{abstract}

\maketitle

\section{Introduction}

A diffeomorphism $f:M\to M$ in a closed Riemannian manifold $M$ is called \emph{partially hyperbolic} if it preserves a continuous $Df$-invariant splitting $$TM=E^s\oplus E^c \oplus E^u$$ such that for some positive integer $\ell>0$ one has that 
\begin{center} $
\max\{\|Df^\ell_xv^s\|,\|Df^{-\ell}_xv^u\|\}< 1$ \hspace{0.2cm} and \hspace{0.2cm}
$\|Df^{\ell}_xv^s\|<\|Df^{\ell}_xv^c\|< \|Df^{\ell}_xv^u\|$
\end{center}
for every $x\in M$ and unit vectors $v^s\in E^s(x)$, $v^c\in E^c(x)$ and $v^u\in E^u(x)$.

In this paper we will always assume that $M$ is a closed (i.e. compact, connected, boundaryless) manifold of dimension 3. The classification of (transitive) partially hyperbolic diffeomorphisms in dimension 3 has attracted a lot of attention since the proposal, due to Pujals (and formalised by Bonatti-Wilkinson, see \cite{BW05}), that they could be classified. We refer the reader to \cite{HP18, BFP} for more details on the general classification problem. 

Recently, many new examples have appeared, as noted in  \cite{BPP16, BGP16, BGHP20}, these examples make  the overall classification rather challenging, but there are many instances where a priori information on the dynamics along the center direction is available. In fact, by adding an assumption of topological neutrality on the dynamics on the center, Pujals' conjecture was established in \cite{BZ19}, this assumption forces that dynamics along the center does not expand or contract when iterating forward. Stronger classification results (smooth classification) have been produced by adding additional hypothesis on the derivative of the diffeomorphism, see \cite{BZ19, CPRH21, MM22, AM23}. 

In this paper we take a step further and we relax the topological neutrality condition to a necessary condition to get the desired classification, thus, extending the result of \cite{BZ19} and showing that the class of transitive partially hyperbolic diffeomorphisms which act quasi-isometrically on the center direction, verify the classification conjecture. We remark that this class was studied in \cite{CP22} in all dimensions where they produced some dynamical consequences of this property along the center.

Let us present the results more precisely. 

A key tool in the study of partially hyperbolic diffeomorphisms is the presence of \emph{invariant foliations} $\W^\sigma$ whose leaves are $C^1$ submanifolds tangent to $E^{\sigma}$ for some $\sigma\in\{s,c,u,cs,cu\}$. We will denote as $\W^\sigma_K(x)$ to the set of points in the leaf $\W^\sigma(x)$ through $x$ whose distance to $x$ is less than or equal to $K$ with the metric induced in the leaves by the ambient Riemannian metric. Note that while the existence of $\W^s$ and $\W^u$ is always guaranteed by \cite{HPS77}, the existence of the other foliations may fail in general.

\begin{definition}(Quasi-isometric center action)
A partially hyperbolic diffeomorphism $f:M\to M$ acts \emph{quasi-isometrically} in a $f$-invariant foliation $\W^c$ if there exist constants $r,R>0$ such that $$f^n(\W^c_r(x))\subset \W^c_R(f^n(x))$$ 
for every $x\in M$ and $n\in \mathbb{Z}$. 
\end{definition}

The goal of this article is to classify the (transitive) partially hyperbolic diffeomorphisms acting quasi-isometrically in a center foliation $\W^c$ in $\dim(M)=3$. We remark that the known examples of such systems are those called \emph{skew-products} and those for which there is an iterate which is a \emph{discretised Anosov flow}. These will be defined precisely in \S~\ref{s.prelim}.

A diffeomorphism $f$ is said to be \textit{transitive} if there exists a point whose orbit under $f$ is dense in the entire space. In other words, for transitivity, there is a point $x$ such that the set $\{ f^n(x) \mid n \in \mathbb{Z} \}$ comes arbitrarily close to any point in the space. More generally, we say that $f$ is \emph{chain-recurrent} if for every non-empty open set $U$ such that $f(\overline{U})\subset U$ we have that $U=M$ (we refer the reader to \cite{CPNotas15} for some equivalent definitions).

There is a non-transitive example of a quasi-isometric center action which does not falls in the above categories. It is constructed in \cite{BPP16} from a non-transitive Anosov flow. In \S~\ref{ss.transitive} we explain where the transitivity assumption is used and what are the possible variations of our result without the assumption (involving the concept of \emph{collapsed Anosov flows} from \cite{BFP}).

The main result of this article is:

\begin{teoA}\label{teoA}
Suppose $f:M \to M$ is a partially hyperbolic diffeomorphism acting quasi-isometrically on a center foliation $\W^c$. Then:
\begin{enumerate}
\item If $f$ is chain-recurrent, then either $f$ is a skew-product or an iterate of $f$ is a discretized Anosov flow.
\item  If $\W^c$ has a dense leaf, then an iterate of $f$ is a discretized Anosov flow.
\end{enumerate}
\end{teoA}

The same conclusion can be obtained in the case $f$ \emph{acts quasi-isometrically} on a \emph{branching center foliation} $\W^c$ (meaning that $\W^c$ is the intersection of two \emph{branching foliations} $\W^{cs}$ and $\W^{cu}$ as given by Burago-Ivanov \cite{BI08}). See Theorem \ref{thmA'} in Section \ref{sectionqionbranching}.  

The proof of Theorem \ref{teoA} requires showing a result (Theorem \ref{thmBG}) about self-orbit equivalences of Anosov flows that may be interesting on its own which was first noticed in \cite{BG2021}; a proof appears in the appendix.

{\small \emph{Acknowledgements:} We want to thank Yuri Lima and Mauricio Polletti for organizing an event in Fortaleza in January 2024 on which the authors met. We thank Yi Shi for discussions that made us aware that the statement of Theorem \ref{thmBG} appeared in \cite{BG2021}. We are deeply thankful to Thomas Barthelm\'e and Andrey Gogolev for letting us include the proof of  Theorem \ref{thmBG} in the appendix. While we were finishing the paper, Ziqiang Feng communicated to us that he had proved a similar result with different techniques \cite{F24}, we thank him for useful discussions. }

\section{Some preliminaries}\label{s.prelim}

Let us first give some definitions. We will say that a partially hyperbolic diffeomorphism is a \emph{skew-product} if it preserves a center foliation by circles of uniformly bounded length (i.e. there is a foliation $\W^c$ tangent to $E^c$ and $f$-invariant so that every leaf is compact and such that the length of every leaf is bounded by a uniform constant). These have been well studied in \cite{BW05,B13,BB16} and up to finite lift the center foliation is actually a circle bundle over $\TT^2$ and the induced dynamics is the quotient is an Anosov homeomorphism. 

A \emph{discretised Anosov flow} is a partially hyperbolic diffeomorphism $f: M \to M$ so that there exists a continuous flow $\varphi_t: M \to M$ with $C^1$-orbits and a continuous map $\tau : M \to \RR_{>0}$ such that $f(x) = \varphi_{\tau(x)}(x)$ for every $x\in M$. It follows a posteriori that the flow $\varphi_t$ is a (topological) Anosov flow whose orbits are tangent to the $E^c$ bundle (and therefore form an invariant center foliation). See \cite{M23} for more details and equivalences. 

In the case where $\pi_1(M)$ is virtually solvable, Theorem \ref{teoA} follows immediately from the classification theorem given in \cite[Main theorem]{HP15}. See Lemma \ref{lemmacasovirtsolv}. 

In the case where $\pi_1(M)$ is not virtually solvable, we will show that $f$ needs to be a discretized Anosov flow. We will first notice that the assumptions allow us to assume that there are invariant foliations for $f$. In fact, the following result holds (see \cite[Proposition 3.7]{M23}): 

\begin{thm}\label{teo-santi}
Let $f: M \to M$ be a partially hyperbolic diffeomorphism of a closed 3-manifold acting quasi-isometrically on a center foliation $\W^c$. Then, $f$ is \emph{dynamically coherent}, that is, there are $f$-invariant foliations $\W^{cs}$ and $\W^{cu}$ tangent respectively to $E^s \oplus E^c$ and $E^c \oplus E^u$ such that $\W^c = \W^{cs}\cap \W^{cu}$. 
\end{thm}

To these invariant foliations we will apply the following result (see \cite[Theorem 11.2]{FP23}): 

\begin{thm}\label{thmFP}
Let $f$ be a partially hyperbolic diffeomorphism in a closed 3-manifold $M$, with
$\pi_1(M)$ not virtually solvable, and such that $f$ admits branching foliations, center
stable ($\W^{cs}$) and center unstable ($\W^{cu}$), both of which have Gromov hyperbolic leaves. Suppose that the center leaf space of $f$ is Hausdorff. Then $f$ is a quasigeodesic partially hyperbolic diffeomorphism, that is, when seen in the universal cover $\widetilde{M}$ of $M$, every center leaf $\ell \in \widetilde{\W}^c$ is a quasi-geodesic inside its leaf of $\widetilde{\W}^{cs}$ and $\widetilde{\W}^{cu}$. 
\end{thm}

In \S~\ref{sectionqionbranching} we will discuss branching foliations, but let us point out that when $f$ is dynamically coherent (which by Theorem \ref{teo-santi} is the case for the proof of Theorem \ref{teoA}) then $f$ admits branching foliatons (which are in fact true foliations). 

The strategy of our proof will be to show that, on the one hand, leaves of the center-stable and center-unstable foliations are Gromov hyperbolic, and on the other hand, the center foliation when restricted to leaves of each of the foliations is Hausdorff. 

Then using Theorem \ref{thmFP} and \cite{BFP} we know that up to finite cover we can assume that $f$ is what is called a \emph{collapsed Anosov flow} (this will be explained in more detail at the end of the next section) and therefore showing that $f$ is a discretised Anosov flow amounts to showing that an iterate of a \emph{self-orbit equivalence} of the Anosov flow is the trivial self orbit equivalence. This follows from the following result that was first noted in \cite{BG2021}:

\begin{thm}[Barthelmé-Gogolev \cite{BG2021}]\label{thmBG}
Let $h:M\to M$ be a self-orbit equivalence of a transitive Anosov flow $\phi_t$. Suppose that for every periodic orbit $\gamma$ of $\phi_t$ there exists $n>1$ such that $h^n(\gamma)=\gamma$ (as a set). Then there exists $N>0$ such that $h^N$ is trivial (that is, $h^N(o)=o$ for every orbit $o$ of $\phi_t$).
\end{thm}

In other words, Theorem \ref{thmBG} says that every self-orbit equivalence for which every periodic orbit of the flow is periodic, is in fact a periodic self orbit equivalence. We give a proof of this theorem in Appendix \ref{appendix} with the permission of Barthelm\'e and Gogolev.

\section{Quasi-isometric action on a center foliation}

\subsection{Proof of Theorem \ref{teoA}}

As a standing assumption along this subsection $f$ will be a partially hyperbolic diffeomorphism in a closed $3$-manifold $M$ acting quasi-isometrically on a center foliation $\W^c$.

Let us see first that we can take finite lifts and iterates if needed:

\begin{lemma}\label{lemmafinitelifts}

If $f^k$ for some $k>0$ has a finite lift which is a partially hyperbolic skew-product or a discretized Anosov flow, then $f^k$ is itself a partially hyperbolic skew-product or a discretized Anosov flow.

\end{lemma}
\begin{proof}
Suppose $N$ is a finite cover of $M$ and $g:N\to N$ is a lift of $f^k$ for some $k>0$. 

The foliation $\W^c$ lifts to a $g$-invariant foliation $\W^c_g$. Since the map $f$ acts quasi-isometrically on $\W^c$, so does $f^k$, and it is immediate from this that also $g$ acts quasi-isometrically on $\W^c_g$.

Recall that by \cite[Proposition 3.7 and Theorem B']{M23} every partially hyperbolic diffeomorphism acts quasi-isometrically on at most one center foliation.

In case $g$ is a partially hyperbolic skew-product then $g$ admits a center foliation by compact leaves of uniform length. It follows by the aforementioned uniqueness result that this foliation has to be $\W^c_g$. Since every leaf of $\W^c_g$ is compact with uniform length, the same is the case for $\W^c$. Then $f^k$ is a partially hyperbolic skew-product.

Suppose now that $g$ is a discretized Anosov flow. Recall that \cite[Proposition 3.3]{M23} states that a partially hyperbolic diffeomorphism $h$ is a discretized Anosov flow if and only if there exists a one-dimensional center foliation $\W^c_h$ and a constant $R>0$ such that $h(x)\in \W^c_h(x,R)$ for every point $x$, where $\W^c_h(x,R)$ denotes the ball in $\W^c_h$ of center $x$ and radius $R>0$.  Moreover, $h$ acts quasi-isometrically on $\W^c_h$ (see \cite[Remark 3.5]{M23}). 

Since $g$ acts quasi-isometrically on $\W^c_g$ and $g$ acts quasi-isometrically on at most one center foliation there exists $R>0$ such that $g(x)\in \W^c_g(x,R)$ for every $x\in N$. It is immediate then that $f^k(x)\in \W^c(x,R)$ for every $x\in M$. So $f^k$ is a discretized Anosov flow by the converse of \cite[Proposition 3.3]{M23}.
\end{proof}

The following is an immediate consequence of \cite[Main theorem]{HP15}.

\begin{lemma}\label{lemmacasovirtsolv}
If $\pi_1(M)$ is virtually solvable, then an iterate of $f$ is either a discretized Anosov flow or a partially hyperbolic skew-product. 
\end{lemma}
\begin{proof}
By \cite[Main theorem]{HP15}, a finite iterate and finite lift of $f$ is a map $g:N\to N$, in a finite lift $N$ of $M$, such that $g$ is either a \emph{Derived-from-Anosov} map, a partially hyperbolic skew-product or a discretized Anosov flow.

One can lift $\W^c$ to a $g$-invariant center foliation. It is immediate that $g$ acts quasi-isometrically in this center foliation. A Derived-from-Anosov map cannot act quasi-isometrically on a center foliation because by assumption is semiconjugate to an Anosov with non-bounded center leaves, and thus it expand the center leaves either for forward or backward iterates. See \cite{P14}.  

By Lemma \ref{lemmafinitelifts} one concludes.
\end{proof}

In view of Lemma \ref{lemmafinitelifts} and Lemma \ref{lemmacasovirtsolv} above, we assume from now on that $\pi_1(M)$ is not virtually solvable, that $E^s$, $E^c$ and $E^u$ are orientable and that $f$ preserves their orientations.

We must consider two different hypothesis: (1) $f$ is chain-recurrent or (2) $\W^c$ has a dense leaf. Lemma \ref{Wcsminimal_if_fchiantrans} and Lemma \ref{Wcsminimal_if_Wctrans} below show that (1) or (2) implies that (3) $\W^{cs}$ and $\W^{cu}$ are minimal foliations. In the rest of the section we will show that (3) implies that an iterate of $f$ is a discretized Anosov and this will conclude the proof of the theorem.

\begin{lemma}\label{Wcsminimal_if_fchiantrans}
If $f$ is chain-recurrent then $\W^{cs}$ and $\W^{cu}$ are minimal foliations.
\end{lemma}
\begin{proof} Recall that every minimal sets of a codimension-one foliation on a closed manifold is either the whole manifold, a compact leaf or a third type of minimal set called \emph{exceptional}. The number of exceptional minimal sets is known to be finite (see for example \cite[Theorem 4.1.3]{HH87}). 

Suppose by contradiction that $\W^{cu}$ is not minimal (for the foliation $\W^{cs}$ the reasoning is analogous). By \cite{RHRHU11} there are no compact leaves of $\W^{cu}$. It follows that $\W^{cu}$ has a finite number of minimal sets (which are also exceptional). Denote them $\{\Lambda_1^{cu},\ldots,\Lambda^{cu}_k\}$ and let $\Lambda^{cu}$ denote its union.

Since $f$ permutes the elements of $\{\Lambda_1^{cu},\ldots,\Lambda^{cu}_k\}$ it follows that $f(\Lambda^{cu})=\Lambda^{cu}$. Let  $U$ be the open set $U:=\bigcup_{x\in \Lambda} \W^s_\epsilon(x)$ for some $\epsilon>0$. By the uniform contraction along stable leaves it follows that some positive iterate of $f$ sends the closure of $U$ strictly inside $U$. That is, $U$ is a \emph{trapping region} for this positive iterate of $f$ and this positive iterate is therefore not chain-recurrent. It is an exercise that the non-existence of trapping regions is equivalent to chain-recurrency, and that chain-recurrency is preserved by finite iterates (see for example \cite{CPNotas15}). We obtain that $f$ is not chain-recurrent and this gives us a contradiction.
\end{proof}

By \cite[Proposition 3.7]{M23}, the leaves of $\W^{cs}$ and $\W^{cu}$ satisfy that $\W^{cs}(x)=\bigcup_{y\in \W^c(x)}\W^s(y)$ and $\W^{cu}(x)=\bigcup_{y\in \W^c(x)}\W^u(y)$ for every $x\in M$. By \cite{RHRHU11} no leaf of $\W^{cs}$ or $\W^{cu}$ is compact. It follows that for every $x\in M$, if $\W^c(x)$ intersects $\W^s(x)$ only in $x$ then $\W^{cs}$ is topologically a plane, and if not, then $\W^{cs}(x)$ is topologically a cylinder or a Möbius band.

We denote by $\widetilde{\W}^{cs}$, $\widetilde{\W}^{cu}$ and $\widetilde{\W}^c$ the lifts of $\W^{cs}$, $\W^{cu}$ and $\W^c$ to the universal cover $\widetilde{M}$, respectively.

\begin{lemma}\label{Wcsminimal_if_Wctrans}
If $\W^c$ has a dense leaf then $\W^{cs}$ and $\W^{cu}$ are minimal foliations.
\end{lemma}
\begin{proof}
Suppose by contradiction that there exists a leaf in $\W^{cs}$ such that its closure $\Lambda\subset M$ is not dense in $M$. We are going to show that every leaf of $\W^{cs}$ cannot be dense in $M$. As a consequence, no orbit of $\W^c$ can be dense either and this will give us a contradiction.

Let us consider $\mathcal{L}^{cs}$ the leaf space of the foliation $\widetilde{\W}^{cs}$ in $\widetilde{M}$. Note that $\pi_1(M)$ acts on $\mathcal{L}^{cs}$. A leaf $L\in \widetilde{\W}^{cs}$ has a dense $\pi_1(M)$-orbit in $\mathcal{L}^{cs}$ if and only if it is the lift of a leaf from $\W^{cs}$ which is dense in $M$.

Let $\widetilde{\Lambda}$ be the lift of $\Lambda$ to $\widetilde{M}$. Note that since $\widetilde{\Lambda}$ is saturated by leaves of  $\widetilde{\W}^{cs}$ it can be seen as subset of the leaf space $\mathcal{L}^{cs}$. Let $I\subset \mathcal{L}^{cs}$ denote an interval in $\mathcal{L}^{cs}$ whose endpoints $L_-$ and $L_+$ lie in $\widetilde{\Lambda}$ but whose interior is disjoint from $\widetilde{\Lambda}$.

Let $L\in \widetilde{\W}^{cs}$ be a leaf in $I$. Note that for every $\gamma\in\pi_1(M)$ such that $\gamma L$ lies in $I$ it follows that $\gamma I$ must be equal to $I$. This is because $\widetilde{\Lambda}$ is invariant by $\gamma$. Then $\gamma$ either fixes or permutes the endpoint $L_-$ and $L_+$.

On the other hand, the elements of $\pi_1(M)$ which leave invariant $L_-$ (or $L_+$) form a subgroup of $\pi_1(M)$ which is either trivial or isomorphic to $\mathbb{Z}$ because every leaf of $\W^{cs}$ is either a plane, a cylinder or a Möbius band.

One obtains that the orbit by $\pi_1(M)$ of $L$ in $I$ has it (possibly empty) accumulation points lying in $\{L_-,L+\}$. Thus the orbit of $L$ cannot be dense in $I$ and as a consequence it is not dense in $\mathcal{L}^{cs}$.
\end{proof}

In view of Lemma \ref{Wcsminimal_if_fchiantrans} and Lemma \ref{Wcsminimal_if_Wctrans}, we will suppose from now on that $\W^{cs}$ and $\W^{cu}$ are minimal foliations. We will show from this that an iterate of $f$ needs to be a discretized Anosov flow.  

Note that by Novikov's theorem each leaf of $\widetilde{\W}^{cs}$ and $\widetilde{\W}^{cu}$ is topologically a plane.

\begin{lemma}\label{lemmaGromovhyp} The leaves of $\widetilde{\W}^{cs}$ and $\widetilde{\W}^{cu}$ are Gromov hyperbolic.
\end{lemma}
\begin{proof}
By Candel's uniformization theorem (see \cite{C93})  if there is no transverse invariant measure for $\W^{cs}$, then, the leaves of $\widetilde{\W}^{cs}$ are Gromov hyperbolic. In case $\W^{cs}$ admits a transverse invariant measure, \cite[Theorem 5.1]{FP2022} shows that the leaves of $\widetilde{\W}^{cs}$ need to be Gromov hyperbolic. For $\widetilde{\W}^{cu}$ the argument is analogous.
\end{proof}

Given a leaf $L$ in $\widetilde{\W}^{cs}$ or $\widetilde{\W}^{cu}$, the leaves of $\widetilde{\W}^c$ subfoliate $L$. We denote this foliation $L_{\widetilde{\W}^c}$ and we denote by  $\mathcal{O}^c_L$ to its leaf space.

\begin{lemma}\label{lemmaleafspaceHsdff} The leaf space $\mathcal{O}^c_L$ is Hausdorff for every leaf $L$ in $\widetilde{\W}^{cs}$ or $\widetilde{\W}^{cu}$.
\end{lemma}
\begin{proof}

Suppose without loss of generality that $L$ is a leaf in $\widetilde{\W}^{cs}$. To show that $\mathcal{O}^c_L$ is Hausdorff it is enough to show that $L_{\widetilde{\W}^c}$ does not have a pair of \emph{non-separated leaves}. That is, that there is no pair of leaves $W$ and $W'$ such that for every pair of transversals $I$ to $W$ and $I'$ to $W'$ there exists a leaf of $L_{\widetilde{\W}^c}$ intersecting both $I$ and $I'$.

Suppose by contradiction that there exist non-separated leaves $W$ and $W'$. Let $x\in W$ and $x'\in W'$. Let $\eta_s$ and $\eta_s'$ be two stable arcs through $x$ and $x'$, respectively. Since $W$ and $W'$ are non-separated there exists a center arc of bounded length $\eta_c$ intersecting both $\eta_s$ and $\eta_s'$. Let $y$ and $y'$ denote these two points of intersection, respectively. See Figure \ref{fig:Hausdorff1}.

\begin{figure}[h] 
	\centering
	\includegraphics[width=0.5\textwidth]{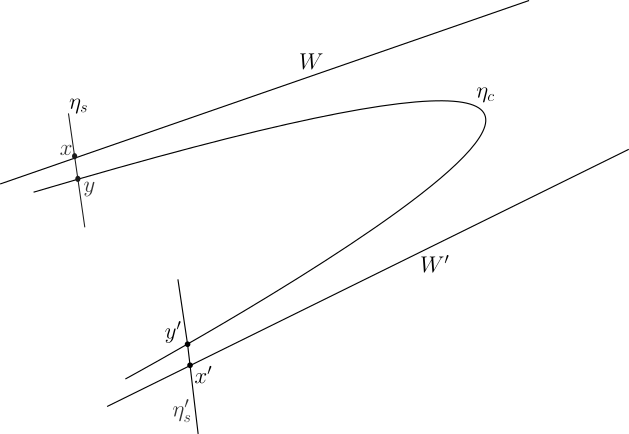}
	\caption{Construction of the arc $\eta_c$ and the points $y,y'$.} 
	\label{fig:Hausdorff1} 
\end{figure}

Let $\widetilde{f}$ be a lift of $f$. Because of the quasi-isometric action on the center the length of $\widetilde{f}^n(\eta_c)$ is bounded uniformly on $n$, for every $n\geq 1$. Say by a constant $R>0$.

By continuity of the foliation $\widetilde{\W}^c$ in the compact manifold $M$ there exist $\epsilon>0$ such that the following is satisfied: For any pair of points $z$ and $z'$ joined by a center arc of length less than $R$, if $w$ is a point in $\widetilde{\W}^s_\delta(z)$ then the center leaf through $w$ intersects $\widetilde{\W}^s(z')$.

By taking $n$ large enough one obtains that $\tilde{f}^n(\eta_s)$ is contained in $\widetilde{\W}^s_\delta(\tilde{f}^n(y))$. It follows that $\tilde{f}^n(W)$ intersects $\widetilde{\W}^s(\tilde{f}^n(y'))$. See Figure \ref{fig:Hausdorff2}. 
\begin{figure}[h] 
	\centering
	\includegraphics[width=0.5\textwidth]{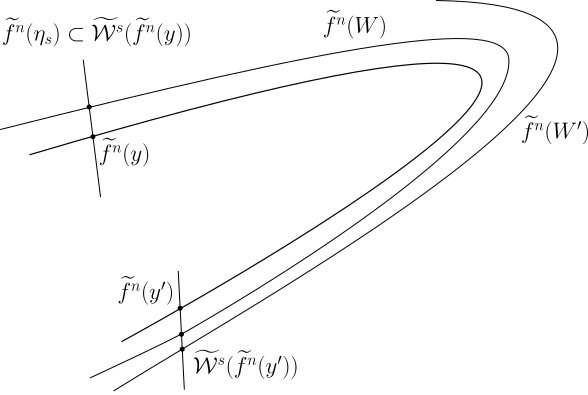}
	\caption{Continuity of holonomy forces $\tilde{f}^n(W)$ to intersect  $\tilde{\W}^s(\tilde{f}^n(y'))$.}
	\label{fig:Hausdorff2}
\end{figure}

Iterating backwards $n$ times one obtains that $W$ intersect $\tilde{\W}^s(y')$. Note that $W'$ also intersects $\tilde{\W}^s(y')$ in $x'$. One can consider $I$ and $I'$ disjoint subarcs of $\tilde{\W}^s(y')$ such that $I\cap W\neq \emptyset$ and $I'\cap W'\neq \emptyset$. No center leaf can intersect $I$ and $I'$ because it would intersect $\tilde{\W}^s(y')$ twice. It follows that $W$ and $W'$ are separated. A contradiction.
\end{proof}

Lemma \ref{lemmaGromovhyp} and Lemma \ref{lemmaleafspaceHsdff} allow us to use Theorem \ref{thmFP} to obtain that the leaves of $\widetilde{\W}^c$ are uniformly quasigeodesics inside the leaves of $\widetilde{\W}^{cs}$ and $\widetilde{\W}^{cu}$. This is what is called a \emph{quasigeodesic partially hyperbolic diffeomorphism} in \cite{BFP}. 

Since the bundles  $E^s$ and $E^u$ are orientable, it follows from \cite[Theorem D]{BFP} that $f$ is a \emph{leaf space collapsed Anosov flow}. In particular, since $f$ is dynamically coherent, the foliation $\W^c$ is given by the flow lines of a topological Anosov flow $\phi^c_t:M\to M$. 

Indeed, thanks to our assumptions, we know that leaves of $\W^{cs}$ and $\W^{cu}$ are dense, and therefore, we deduce that the topological Anosov flow $\phi^{c}_t$ is transitive (and therefore orbit equivalent to a true Anosov flow by \cite{S21}, though we will not use it).

It follows that $f$ is a \emph{self-orbit equivalence} of the topological Anosov flow $\phi^c_t:M\to M$. That is, $f$ sends orbits of $\phi^c_t:M\to M$ to orbits of itself, and preserving their orientations.

It is a well-known fact that for every topological Anosov flow, for every $R>0$ the number of periodic orbits of length less than $R$ is finite.

Since $f$ acts quasi-isometrically on $\W^c$, then for every periodic orbit $\gamma$ of  $\phi^c_t$ there exists $R_\gamma>0$ such that the length of $f^n(\gamma)$ is less than $R_\gamma>0$ for every $n \in \mathbb{Z}$. As the number of periodic orbits of length less than $R_\gamma>0$ is finite one obtains that $f^k(\gamma)=\gamma$ for some $k> 1$ depending on $\gamma$.

That is, the self-orbit equivalence $f$ has the property that every periodic orbit of $\phi^c_t$ is periodic by $f$. By Theorem \ref{thmBG} one obtains that there exists $K>0$ such that $f^K$ is the trivial self-orbit equivalence. That is, $f^K$ leaves invariant every orbit of $\phi^c_t$.

By \cite[Proposition 5.26]{BFP} (see also \cite[Theorem 1.2]{M23}) one obtains that $f^K$ is a discretized Anosov flow. This completes the proof of Theorem \ref{teoA}. 

\subsection{The transitivity assumption}\label{ss.transitive}

We close this section by making some comments on the transitivity assumption. As we mentioned in the introduction, examples in \cite{BPP16} show that the same statement cannot hold without assuming some form of transitivity. Still, one can wonder if every partially hyperbolic in dimension 3 acting quasi-isometrically on a center foliation must be a \emph{collapsed Anosov flow}. 

We remark that we have used the transitivity in two points in the proof. The crucial one is to show that the flow $\phi^{c}_t$ is transitive to be able to apply Theorem \ref{thmBG}. Theorem \ref{thmBG}  is false without the transitivity assumption (as the examples in \cite{BPP16} show). The other place was to get the minimality of the $\W^{cs}$ and $\W^{cu}$ foliations, which with the help of \cite[Theorem 5.1]{FP2022} allowed to show the Gromov hyperbolicity of center stable and center unstable leaves. 

We believe that the same should hold for every partially hyperbolic diffeomorphism, namely, that if $f: M \to M$ is a (dynamically coherent) partially hyperbolic diffeomorphism on a closed 3-manifold with non-virtually solvable fundamental group, then, the leaves of $\W^{cs}$ and $\W^{cu}$ are Gromov hyperbolic. We have not pursued this problem, that would give that such diffeomorphisms must all be collapsed Anosov flows.

\section{Quasi-isometric action on a branching center foliation}\label{sectionqionbranching}

By Burago-Ivanov \cite{BI08}  under orientability assumptions there exist branching $cs$ and $cu$ foliations. We say that $\W^c$ is a \emph{branching center foliation} if it is the intersection of a branching $cs$ and a branching $cu$ foliation. In this technical section, we shall assume familiarity with branching foliations (see also \cite[\S 3]{BFP} for a presentation in this precise context). 

\begin{definition}
We say that a partially hyperbolic diffeomorphism $f:M\to M$ \emph{acts quasi-isometrically} on a branching center foliation $\W^c$ if there exist constants $r,R>0$ such that for every segment $\eta$  in a leaf of $\W^c$ such that $\length(\eta)\leq r$ one has that $$\length(f^n\eta)\leq R$$ 
for every $n\in \mathbb{Z}$. 
\end{definition}

We say that a branching foliation $\W$ is a \emph{true foliation} if no branching point exists. The goal of this section is to show the following:

\begin{teoAprima}\label{thmA'}
Suppose $f:M\to M$ is a partially hyperbolic diffeomorphism acting quasi-isometrically on a branching center foliation $\W^c$. Then $\W^c$ is a true foliation and:
\begin{enumerate}
\item If $f$ is chain-recurrent, then either $f$ is a skew product or an iterate of $f$ is a discretized Anosov flow. 
\item  If $\W^c$ has a dense leaf, then an iterate of $f$ is a discretized Anosov flow.
\end{enumerate}
\end{teoAprima}

To show Theorem \ref{thmA'} it is enough to reduce to the case of Theorem \ref{teoA} because of the following.

\begin{prop}
If $f$ acts quasi-isometrically on a branching center  foliation $\W^c$, then $\W^c$ is a true foliation.
\end{prop}

\begin{proof}
Suppose that $\W^c$ is the intersection of two branching foliations $\W^{cs}$ and $\W^{cu}$. To show that $\W^c$ is a true foliation it is enough to show that $\W^{cs}$ and $\W^{cu}$ are true foliations.

Every leaf $L$ of $\W^{cs}$ or $\W^{cu}$ is subfoliatied (a priori with branching) by the leaves of $\W^c$. Let us denote $\W^{c,L}$ to this foliation. It is enough to show that $\W^{c,L}$ is a true foliation for every $L$ in  $\W^{cs}$ or $\W^{cu}$.

Suppose by contradiction that $\W^{c,L}$ is not a true foliation for some $L$ in $\W^{cu}$ (for $L$ in $\W^{cs}$ the arguments are analogous). There exists a point $x$ in $L$ such that through $x$ there are several local leaves of $\W^{c,L}$. More precisely, one can consider small non-trivial $C^1$ segments $\eta_c$ and $\eta'_c$ contained in leaves of $\W^{c,L}$ and a small non-trivial $C^1$ segment $\eta_u$ in a leaf of $\W^u$ such that $x$ is one of the endpoints of $\eta_c$ and $\eta_c'$, and the other two endpoints of $\eta_c$ and $\eta_c'$ are the (different) endpoints of $\eta_u$. Informally, the strategy will be to use a volume versus length type of argument.

Let $V$ denote the interior of the `triangle' in $L$ whose sides are $\eta_c$, $\eta'_c$ and $\eta_u$, and let $\overline{V}$ denote its closure (which is equal to the union of $V$ with $\eta_c$, $\eta'_c$ and $\eta_u$). Note that $\overline{V}$ is equal to the union of every center segment $\gamma\subset \overline{V}$ joining $x$ to a point in $\eta_u$. There exists $r>0$ an upper bound for the length of every such a center segment. By the quasi-isometric action on $\W^c$ there exists $R>0$ such that the length of $f^n(\gamma)$ is bounded by $R>0$ for every integer $n$ and every such a center segment $\gamma$.

Note that $V$ has interior. Let $\gamma_u$ be a non-trivial segment of $\W^u$ contained in $V$. One has that $f^n(\gamma_u)$ lies in the ball of center $f^n(x)$ and radius $R$ for every integer $n$.

Let $L_n$ denote the leaf $f^n(L)$ and $\widetilde{L_n}$ its universal cover for every $n\in \mathbb{Z}$. Note that $\widetilde{L_n}$ is topologically a plane. Let $\W^{u,\widetilde{L_n}}$ denote the lift of the restriction of $\W^u$ to $L_n$ for every $n\in \mathbb{Z}$. 

Since one can cover the closed manifold $M$ by a finite number of compact foliation boxes of $\W^u$, then the every ball of radius $R$ in a leaf $L_n$ can be covered by a fixed finite number (independent of $n$) of foliation boxes of $\W^{u,\widetilde{L_n}}$ whose maximum length of plaque is uniformly bounded (independent of $n$). Since the length of $f^n(\gamma_u)$ tends to infinity with $n$, it follows that a lift of $f^n(\gamma_u)$ to $L_n$ (which is contained in a ball of radius $R$) must intersect two times the same center arc in $L_n$. As a consequence, since $L_n$ is topologically a plane then $\W^{u,\widetilde{L_n}}$ must have a singularity in $L_n$, which is impossible. 
\end{proof}

\appendix

\section{A result on self-orbit equivalences of Anosov flows}\label{appendix}

In this appendix we give a proof of Theorem \ref{thmBG} that came out in discussions with T. Barthelm\'e and A. Gogolev. They were kind to allow us reproduce the argument here.

Let $\phi_t$ be a transitive Anosov flow in dimension 3. Suppose that $f$ is a self-orbit equivalence of $\phi_t$ such that for every periodic orbit $o$ of $\phi_t$ there exists $n\geq 1$ such that $f^n(o)=o$. Let us see that there exists $N\geq 1$ such that $f^N$ is the trivial self-orbit equivalence. That is, such that $f^N(o)=o$ for every orbit $o$. Note that there is no loss in generality in assuming that every bundle is orientable and that $f$ preserves orientation along the weak foliations since this can be achieved by taking a finite lift and iterate.  

The proof will be carried out in the bifoliated plane $(\mathcal{O}_\phi,\mathcal{G}^s,\mathcal{G}^u)$ of $\phi_t$. We will assume some familiarity with it (see for example \cite{BFM22}).  Note that in the case when $\phi_t$ is orbit equivalent to a suspension the result is direct since these admit finitely many self-orbit equivalence classes, thus we will assume throughout that $M$ does not have virtually solvable fundamental group.

Given $F$ a lift to $\widetilde{M}$ of $f^n$ for some integer $n$, note that $\pi_1(M)$ and $F$ act both on $(\mathcal{O}_\phi,\mathcal{G}^s,\mathcal{G}^u)$. Let us denote by $G(F)$ the group of transformations of $(\mathcal{O}_\phi,\mathcal{G}^s,\mathcal{G}^u)$ generated by $\pi_1(M)$ and $F$. As mentioned before, we can and we will assume that these actions are orientation preserving. 

We say that an orbit $o$ in $\mathcal{O}_\phi$ is a \emph{periodic orbit} if it is the lift of a periodic orbit of $\phi_t$.

\begin{lemma}\label{lemma1appendix}
Suppose $F$ is a lift of $f^n$ for some integer $n$. If $\widetilde{o}\in \mathcal{O}_\phi$ is a periodic orbit then $G(F).\widetilde{o}$ is a discrete subset of $\mathcal{O}_\phi$.
\end{lemma}
\begin{proof}
Let $o$ be the periodic orbit whose lift is $\widetilde{o}$. Let $n>1$ be such that $f^n(o)=o$. Since the set $\bigcup_{0\leq i \leq n-1} f^i(o)$ is the union of finitely many periodic orbits it must intersect every small disc transverse to $\phi_t$ in a finite number of points. The orbit of $\widetilde{o}$ by $G(\widetilde{f})$ is equal to the lift of $\bigcup_{0\leq i \leq n-1} f^i(o)$ to $\widetilde{M}$. Thus it cannot accumulate at any point of $\mathcal{O}_\phi$.
\end{proof}

We say that $R\subset\mathcal{O}_\phi$ is a \emph{rectangle} if it is the compact region determined by the $\mathcal{G}^s$ and $\mathcal{G}^u$ leaves of two distinct points $x,y \in \mathcal{O}_\phi$ satisfying that $\mathcal{G}^s(x)\cap \mathcal{G}^u(y) \neq \emptyset$ and $\mathcal{G}^u(x)\cap \mathcal{G}^s(y) \neq \emptyset$. Note that inside $R$ the foliations $\mathcal{G}^s$ and $\mathcal{G}^u$ have global product structure.

\begin{lemma}\label{lemma2appendix}
Let $R$ be a rectangle whose corners are fixed by a lift $F$ of $f^n$ for some integer $n$. Then every point in $R$ is fixed by $F$.
\end{lemma}
\begin{proof}
Note that $R$ is sent by $F$ to a rectangle with the same corners. It follows that $R$ is sent to itself, that is, $R$ is invariant by $F$. Since $F$ preserves orientation, then it fixes all the corners too. 

Suppose by contradiction that $F:R\to R$ is not the identity map. There has to exist a subinterval $I$ of one of the sides of $R$ that is invariant by $F$ and such that $F:I\to I$ is conjugate to a nontrivial translation. Suppose without loss of generality that $I$ is a subinterval of a side contained in a leaf of $\mathcal{G}^s$.

Points in $\mathcal{O}_\phi$ corresponding to periodic orbits of $\phi_t$ are dense in $\mathcal{O}_\phi$ since $\phi_t$ is transitive. Thus there exists one of them, let us call it $\widetilde{o}$, such that $\mathcal{G}^u(\widetilde{o})$ intersects $I$. Since $F$ is a translation in $I$ then every point in the orbit of $\widetilde{o}$ by $F$ must lie in a different leaf of  $\mathcal{G}^u$. In particular, the orbit of $\widetilde{o}$ by $F$ must be infinite. Since this orbit is contained in the compact subset $R$ this contradicts Lemma \ref{lemma1appendix}.

\end{proof}

\begin{lemma}\label{lemma3appendix}
There exists a rectangle $R$ whose corners that are fixed by a lift of $f^n$ for some $n \geq 1$.
\end{lemma}
\begin{proof}
Suppose that $\phi_t$ is not a suspension. Let $x$ be a corner of a lozenge such that $\gamma.x=x$ for some non trivial element $\gamma \in \pi_1(M)$. Let $F$ be a lift of $f^n$ for some $n\geq 1$ such that $F(x)=x$.

We claim that there exists $z\in \mathcal{G}^s_+(x)$ such that $\gamma^m\circ F^k (z)=z$ for some integers $k\geq 1$ and $m$. Indeed, by contradiction, suppose this does not happens. Let $w$ be a point in $\mathcal{G}^s_+(x)$ and let $I$ denote the segment in $\mathcal{G}^s_+(x)$ joining $w$ and $\gamma (w)$. Consider $o$ a periodic orbit in the interior of the lozenge such that $\mathcal{G}^u$ intersects $I$, say that in a point $w'$. For every $k\geq 1$ we can consider $m$ such that $\gamma^m\circ f^k(w')$ lies in $I$. By assumption these points are all different from each other. As a consenquence, for every $k\geq 1$ the points $\gamma^m\circ f^k(o)$ are also all different from each other. These points lie in the compact rectangle whose corners are the intersection of the $\mathcal{G}^u$ leaves though the endpoints of $I$ with $\mathcal{G}^s_+(x)$ and its opposite side of the lozenge. This contradicts Lemma \ref{lemma1appendix} and proves the claim.

So there exists $z\in \mathcal{G}^s_+(x)$ such that $\gamma^m\circ F^k (z)=z$ for some integers $k\geq 1$ and $m$. The map $\gamma^m\circ F^k$ is a lift of $f^{nk}$. By a slight abuse of notation let us rename it by $F$. Then $F$ is a lift of $f^{nk}$ such that $x$ and $z\in \mathcal{G}^s_+(x)$ are fixed by it. By a reasoning analogous to that of the previous claim there must exist a point $z'$ in $\mathcal{G}^u_+(x)$ such that $\gamma^{m'}\circ F^{k'} (z)=z$ for some integers $k'\geq 1$ and $m'$. The map $\gamma^{m'}\circ F^{k'}$ is now a lift of $f^{nkk'}$. Let us rename it by $F$.

It follows that $F$ fixes $x$ and the point $z\in \mathcal{G}^s_+(x)$ and $z'\in \mathcal{G}^u_+(x)$, so it also fixes $y=\mathcal{G}^u(z)\cap \mathcal{G}^s(z')$. The rectangle whose corners are $x$ and $y$ satisfies then the thesis of the lemma for $F$ the lift of $f^{nkk'}$.
\end{proof}

\begin{proof}[Proof of Theorem \ref{thmBG}]
By Lemma \ref{lemma3appendix} there exists a rectangle $R$ with two opposite corners $x$ and $y$ which are fixed by $F$ a lift of $f^n$ for some positive integer $n$. By Lemma \ref{lemma2appendix} every point in $R$ is fixed by $F$. Let us show that every point in $\mathcal{O}_\phi$ needs to be fixed by $F$. As a direct consequence, $f^n$ needs to be the trivial self-orbit equivalence of $\phi_t$.

Suppose $o$ in $\mathcal{O}_\phi$ is a periodic point such that $\mathcal{G}^s(o)$ intersects $R$. The leaf $\mathcal{G}^s(o)$ has to be invariant by $F$ since the points in  $\mathcal{G}^s(o)\cap R$ are fixed by $F$. Since $F$ takes periodic orbits to periodic orbits, and $\mathcal{G}^s(o)$ contains only one, then $o$ needs to be fixed by $F$.

Now, let $z$ be a point such that $\mathcal{G}^s(z)$ intersects the interior of $R$. Since the flow is transitive, periodic orbits are dense in $\mathcal{O}_\phi$, so one can consider a sequence of periodic orbit $o_n$ coverging to $z$. By continuity of $F$ it follows that $z$ is also fixed by $F$.

We have shown that every point in the saturation of $R$ by leaves of $\mathcal{G}^s$ is fixed by $F$. In the same fashion one can show that the saturation by leaves of $\mathcal{G}^u$ of this set is made of fixed points of $F$. Inductively one obtains that every point in $\mathcal{O}_\phi$ is fixed by $F$.
\end{proof}

\bibliographystyle{alpha}
\bibliography{references}

\end{document}